\documentclass{article}

\RequirePackage{amsthm,amsmath,amsfonts,amssymb}
\RequirePackage[numbers]{natbib}
\RequirePackage[colorlinks,citecolor=blue,urlcolor=blue]{hyperref}
\RequirePackage{graphicx}

\usepackage{algorithmic}
\usepackage{algorithm}
\usepackage{multirow}
\usepackage{stmaryrd}
\usepackage[dvipsnames]{xcolor} 
\usepackage{enumitem}
\usepackage{float}
\usepackage{subcaption}
\usepackage[utf8]{inputenc}
\usepackage{authblk}
\usepackage{orcidlink}
\usepackage{indentfirst}
\usepackage[includeheadfoot,margin=2.54cm]{geometry}

\theoremstyle{plain}

\newtheorem{corollary}{Corollary}
\newtheorem{lemma}{Lemma}
\newtheorem{proposition}{Proposition}

\theoremstyle{remark}
\newtheorem{remark}{Remark}

\DeclareMathOperator*{\argmin}{arg\,min}
\DeclareMathOperator*{\argmax}{arg\,max}

\DeclareMathOperator*{\interior}{int}

\DeclareMathOperator*{\domain}{dom}

\DeclareMathOperator*{\minimize}{minimize}

\DeclareMathOperator*{\rk}{rank}

\title{\bf A divergence-based condition to ensure quantile improvement in black-box global optimization}
\author[1,a]{Thomas Guilmeau}
\author[1,b]{Emilie Chouzenoux}
\author[2]{Víctor Elvira}

\affil[1]{Université Paris-Saclay, CentraleSupélec, INRIA, CVN, France}
\affil[ ]{$^{\textrm{a}}$ \texttt{thomas.guilmeau@inria.fr} \orcidlink{0000-0002-8484-6550}}
\affil[ ]{$^{\textrm{b}}$ \texttt{emilie.chouzenoux@centralesupelec.fr} \orcidlink{0000-0003-3631-6093}}
\affil[2]{School of Mathematics, University of Edinburgh, United Kingdom}
\affil[ ]{\texttt{victor.elvira@ed.ac.uk} \orcidlink{0000-0002-8967-4866}}
\date{}

\begin{document}

\maketitle

\begin{abstract}
Black-box global optimization aims at minimizing an objective function whose analytical form is not known. To do so, many state-of-the-art methods rely on sampling-based strategies, where sampling distributions are built in an iterative fashion, so that their mass concentrate where the objective function is low. Despite empirical success, the theoretical study of these methods remains difficult. In this work, we introduce a new framework, based on divergence-decrease conditions, to study and design black-box global optimization algorithms. Our approach allows to establish and quantify the improvement of sampling distributions at each iteration, in terms of expected value or quantile of the objective. We show that the information-geometric optimization approach fits within our framework, yielding a new approach for its analysis. We also establish sampling distribution improvement results for two novel algorithms, one related with the cross-entropy approach with mixture models, and another one using heavy-tailed sampling distributions.
\end{abstract}


\medskip
\noindent
{\bf MSC2020 Subject Classification.} 62F15, 62F30, 62B11, 90C26, 90C30.

\section{Introduction}

Finding the minimizer of a possibly non-convex objective function that is only accessible through a black-box oracle is a challenging, yet important task, which has motivated many works \cite{zhigljavski2008, rios2013}. Given the presence of eventual local minima and the difficulty of evaluating or even approximating gradients, many methods resort to sampling procedures. These rely on evolution strategies to construct proposal distributions \cite{hansen2015}, typically Gaussians, to generate samples close to the minimizers of the objective. Among these methods, one can mention the class of estimation-of-distribution algorithms \cite{larranaga2002}, the cross-entropy algorithm \cite{kroese2006}, or the CMA-ES algorithm \cite{hansen2001, hansen2023}.

A useful perspective to gain theoretical insights is to understand these algorithms as optimization schemes aiming at minimizing an expectation-based reformulation of the original problem over a set of proposal sampling distributions. The resulting reformulated problem can consist in minimizing the expected objective value \cite{glasmachers2010, akimoto2012, malago2015}, or the expected transformed objective value, for some well-chosen transformation \cite{wierstra2008, wierstra2014, brookes2022}. These transformations can also be rank-based \cite{wierstra2008, wierstra2014}. Rank-based transforms only require a ranking of solutions and thus preserve invariance properties with respect to monotonic transformation of the objective function. Algorithms with such invariance properties behave identically on two problems with the same ranking of solutions, allowing to generalize insights from one problem to another \cite{hernando2019}. Invariance properties usually yield better-performing algorithms \cite{hansen2011, beyer2014}. Rank-based transformations rely on reformulations that depend on the retained proposal. In the infinite-population limit, a quantile-based reformulation of the objective function is obtained~\cite{ollivier2017}.

In order to solve the reformulated optimization problem over a set of proposals, many algorithms resort to natural gradient updates. These updates have been proposed in \cite{amari1998}, and consist in a gradient descent step preconditioned by the Fisher information matrix of the proposal. Natural gradients have been used in the context of estimation-of-distribution algorithms \cite{malago2015, brookes2022}, evolution strategies \cite{akimoto2012, wierstra2014}, or discrete permutation problems \cite{ceberio2023}. Natural gradients yield invariance properties with respect to parametrization of the proposals \cite{wierstra2014, ollivier2017}. They are straightforward to be computed when the proposal lies in the exponential family \cite{khan2018}. The natural gradient descent has been used jointly with a (rank-based) quantile-based reformulation of the objective in \cite{ollivier2017,akimoto2013}, leading to the so-called information-geometric optimization (IGO) framework. IGO recovers many existing algorithms, such as the cross-entropy (CE) algorithm \cite{kroese2006}, also based on quantiles, as well as various estimation-of-distribution algorithms~\cite{larranaga2002}.

The theoretical study of the aforementioned natural gradient methods mostly follows two main approaches. The first one is to establish asymptotic convergence of the proposals to a limit proposal distribution well-suited to solve the original problem. This is the approach of \cite{akimoto2012cvgceIGO, beyer2014} showing the convergence of Gaussian proposals used in IGO. Note that these results are established in an infinite sample size regime, and for infinitesimally small step sizes, amounting to continuous time. Their applicability in practical implementations, characterized by non-zero step sizes and stochastic errors, is up to our knowledge still an open problem. The second approach consists in proving an improvement on the reformulated optimization problem at every iteration. Such results are useful in practice as they hold without having to wait for an eventual asymptotic regime to be reached. This is done in \cite{akimoto2012}, in the case of expected objective value minimization over Gaussian proposals, and in \cite{ollivier2017, akimoto2013} for the IGO reformulation of the original problem, although assuming infinitesimally small step-sizes, or proposals within an exponential family. In \cite{ollivier2017, akimoto2013}, the improvement implies a quantile-based improvement on the original problem. This means that a larger fraction of the mass of the proposals will concentrate where the objective function is low as the algorithm iterations progress. We are not aware of any study connecting the two approaches, which can be explained by the very different mathematical tools and paradigms used in both research lines (e.g., continuous versus discrete time).

We follow in this work the second approach, with the aim to establish new improvement results for black-box optimization algorithms. Ideally, one would want proven improvement results that are valid at every iteration, for realistic step sizes, and for a wide variety of proposal models. Typical motivating examples, widely used in black-box global optimization, are heavy-tailed \cite{schaul2011}, and mixture \cite{maree2017, ahrari2017, he2021, ahrari2022}, proposal distributions. However, such proposals would require infinitesimally small step sizes to benefit from the available improvement results from \cite{ollivier2017, akimoto2013}, since they do not form an exponential family. Further, even for proposals within the exponential family, the results from \cite{ollivier2017, akimoto2013} do not quantify the magnitude of the improvements. Therefore, there is still a need for novel wide-ranging criteria to ensure that black-box optimization algorithms improve, and if so, how much, either in terms of reformulated problems or quantiles, that we address in this work.

Our contributions are the following:
\begin{itemize}
    \item We introduce novel divergence-based conditions, measuring, through a Kullback-Leibler (KL) or a Rényi divergence, the discrepancy between a given target distribution and successive proposals. We show that any generic algorithm satisfying those conditions improves in terms of the expectation-based reformulation of the objective at every iteration, and we quantify the improvement. Namely, if the divergence is decreasing between two consecutive proposals, then the decrease in divergence translates into a improvement both in the expectation-based reformulated objective and quantile. It is worthy to emphasize that our results do not depend on the way the next proposal is designed, making our conditions a versatile tool to study evolution strategy algorithms.
    \item We show that the IGO framework fits within the introduced divergence-decrease conditions, illustrating the wide scope of our results. In the case of the IGO reformulation of the objective, we quantify the quantile improvement that comes under our divergence-based conditions. 
    \item We go beyond the scope of the aforementioned IGO works by considering mixture and heavy-tailed proposals. We propose a novel mixture-based algorithm, reminiscent from the mixture-based CE algorithm of \cite[Example 3.2]{kroese2006}, and we show that it fits within our framework. We also propose a new algorithm for Student proposals (having heavier tails than Gaussians and including Cauchy), and we show that it satisfies our divergence-based conditions. We interpret our algorithms as the black-box global optimization counterparts of existing methods in statistics.
\end{itemize}

Let us position our contributions with respect to existing literature. Our results allow to derive new proofs for the quantile improvement in the IGO framework of \cite{ollivier2017, akimoto2013}. As we discuss later in detail, our results are stated for the IGO quantile-based reformulation of the objective. They also hold for other expectation-based reformulations of the objective such as the ones in \cite{wierstra2008, wierstra2014, brookes2022}. We furthermore quantify the improvement in terms of expectation-based reformulation and quantile, yielding more precise results than in \cite{ollivier2017, akimoto2013}. Contrary to existing works, our results can be applied on proposals beyond the exponential family, as we show in several examples, and hold without the stringent assumption of infinitesimal step sizes. 

The paper is organized as follows. We give preliminary notions about black-box global optimization problems, and algorithms to solve them, in Section \ref{section:preliminary}. We then state our main results in Section \ref{section:mainResult}, including our novel conditions for improvement, and new results obtained under these conditions. Finally, we discuss perspectives in Section \ref{section:perspectives}.

\section{Preliminary notions}
\label{section:preliminary}

Let us start with some preliminary notions on sampling-based black-box algorithms for global optimization. These algorithms sample points from a proposal distribution that is updated iteratively so that it concentrates around the solutions of the considered optimization problem. We first present how to reformulate the initial optimization problem into an optimization problem on proposal distributions. Then, we discuss algorithms to solve this resulting problem, making a particular focus on the IGO framework.

\subsection{Problem setting and reformulation}

We consider throughout the paper the generic black-box minimization problem
\begin{equation}
    \label{pblm:generalOpt}
    \minimize_{x \in \mathbb{X}} f(x),
\end{equation}
where $f:\mathbb{X} \rightarrow \mathbb{R}$ may be non-convex and can only be accessed through a black-box that, for a given $x \in \mathbb{X}$, returns the value $f(x)$. The search space $\mathbb{X}$ can be continuous, discrete, or mixed between continuous and discrete variables. We assume the existence of a measure $m$ on $\mathbb{X}$ for some $\sigma$-algebra over $\mathbb{X}$. For instance, if $\mathbb{X} = \mathbb{R}^d$, one can consider the Lebesgue measure, while if $\mathbb{X} = \mathbb{N}^d$, one can take the counting measure.

We focus in our work on algorithms that solve Problem \eqref{pblm:generalOpt} through a sampling-based approach. The aim is to construct a parametric probability distribution $p_{\theta}$ over $\mathbb{X}$ such that $p_{\theta}$ is concentrated around the minimizers of $f$ over $\mathbb{X}$. In this context, one does not search anymore for an optimal point $x \in \mathbb{X}$, but instead for an optimal parameter $\theta \in \Theta$, or alternatively, for an optimal probability distribution $p_{\theta} \in \{ p_{\theta},\,\theta \in \Theta\}$. In the following, we make the standard assumption that the considered proposals $p_{\theta}$ have a density with respect to $m$, also denoted by $p_{\theta}$. One way to transform Problem \eqref{pblm:generalOpt} into a problem over $\Theta$ is to consider the minimization of $\theta \longmapsto \mathbb{E}_{X \sim p_{\theta}} [f(X)]$. This reformulation, which has been studied in \cite{akimoto2012, malago2015} for instance, makes however the resulting algorithm sensitive to transformation of $f$. 

In this work, we focus on an alternative reformulation of Problem~\eqref{pblm:generalOpt}, which has been proposed in the context of IGO \cite{ollivier2017}, and has the advantage of ensuring more invariance properties. Let the $p_{\theta}$-$f$-quantiles at $x \in \mathbb{X}$ be defined by
\begin{align*}
    q_{\theta}^{<}(x) &= \mathbb{P}_{X \sim p_{\theta}} [ f(X) < f(x)],\\
    q_{\theta}^{\leq}(x) &= \mathbb{P}_{X \sim p_{\theta}} [ f(X) \leq f(x)].
\end{align*}
For a given $x \in \mathbb{X}$, $q_{\theta}^{<}$ and $q_{\theta}^{\leq}$ measure the mass that $p_{\theta}$ puts on points that achieve (strictly) better value of $f$ than $x$. Select next a weighting non-increasing function $w: [0,1] \rightarrow \mathbb{R}_+$. The authors of \cite{ollivier2017} then introduced the preference function $W_{\theta}^f : \mathbb{X} \rightarrow \mathbb{R}$ which is defined for any $x\in\mathbb{X}$ as
\begin{equation}
    \label{eq:defW}
    W_{\theta}^f(x) =
    \begin{cases}
        w( q_{\theta}^{\leq}(x)) \text{ if } q_{\theta}^{\leq}(x) = q_{\theta}^<(x),\\
        \frac{1}{q_{\theta}^{\leq}(x) - q_{\theta}^<(x)} \int_{q_{\theta}^<(x)}^{q_{\theta}^{\leq}(x)} w(u)du \text{ otherwise}.
    \end{cases}
\end{equation}
The function $W_{\theta}^f$ is a quantile-based rewriting of $f$ that is invariant under increasing transformation of the objective $f$, as $W_{\theta}^f \equiv W_{\theta}^{\phi \circ f}$ for any increasing function $\phi$ and $\theta \in \Theta$. Also, $W_{\theta}^f$ reflects the behavior of $f$. Indeed, consider $(x,x') \in \mathbb{X}^2$ such that $f(x) \leq f(x')$, $q_{\theta}^{\leq}(x) = q_{\theta}^<(x)$, and $q_{\theta}^{\leq}(x') = q_{\theta}^<(x')$. Then, $W_{\theta}^f(x) \geq W_{\theta}^f(x')$. Under such definitions, given a proposal with parameter $\theta' \in \Theta$, the authors of \cite{ollivier2017} considered the search for a good proposal $p_\theta$ to solve \eqref{pblm:generalOpt} as the maximization of the function $J(\cdot |\theta') : \Theta \rightarrow \mathbb{R}$ defined, for any $\theta \in \Theta$, by
\begin{equation}
    \label{eq:defJ}
    J(\theta | \theta') = \mathbb{E}_{X \sim p_{\theta}}\left[W_{\theta'}^f(X) \right].
\end{equation}
Note that $J(\theta | \theta) = Z_w$ for any $\theta \in \Theta$, with the notation $Z_w = \int _0^1 w(u)du$.

Measuring the quality of a proposal $p_{\theta}$ to solve Problem \eqref{pblm:generalOpt} using quantiles has also been proposed in the framework of the CE method \cite{kroese2006}. Given a proposal $p_{\theta}$ and a scalar $q \in (0,1)$, the CE method relies on $q$-quantiles of $f(X)$ where $X \sim p_{\theta}$, that is, any value $u \in \mathbb{R}$ such that
\begin{equation}
    \label{eq:q-quantileCEdef}
    \mathbb{P}_{X \sim p_{\theta}} [ f(X) \leq u] \geq q \text{ and } \mathbb{P}_{X \sim p_{\theta}}[f(X) \geq u] \geq 1-q.
\end{equation}
Let us denote, as in \cite{akimoto2013}, $Q_{\theta}^q(f)$ as the largest of such values,
\begin{equation}
    \label{eq:defQ_theta}
    Q_{\theta}^q(f) = \sup \{ u \in \mathbb{R} \text{ such that \eqref{eq:q-quantileCEdef} is satisfied}\}.
\end{equation}
If $x$ is sampled from $p_{\theta}$, then $f(x)$ is below $Q_{\theta}^q(f)$ with a probability greater than $q$. Therefore, a good proposal $p_{\theta}$ to solve Problem \eqref{pblm:generalOpt} should be such that $Q_{\theta}^q(f)$ is as low as possible. Remark that $Q_{\theta}^q(f)$ only depends on the current proposal, contrary to $J$.

It is actually possible to relate the behavior of the quantities $J(\theta | \theta')$ from the IGO framework, and $Q_{\theta}^q(f)$ from CE methods, for a particular case of weighting scheme $w$. This is done in \cite{akimoto2013} where the authors showed that an increase in term of $\theta \longmapsto J(\theta | \theta')$ relates to a decrease in terms of $\theta \longmapsto Q_{\theta}^q(f)$, as stated in the lemma hereafter.
\begin{lemma}[Lemma 8 in \cite{akimoto2013}]
    \label{lemma:improvementJ_impliesQuantileImprovement}
    Consider the weighting function $w(u) = \delta_{u \leq q}(u)$ with $q \in (0,1)$. If $(\theta, \theta') \in \Theta^2$ satisfies the increase condition
    \begin{equation}
        \label{eq:strictIncreaseJ}
        J(\theta | \theta') > J(\theta' | \theta') = Z_w,
    \end{equation}
    then we have $Q_{\theta}^q(f) \leq Q_{\theta'}^q(f)$. If further, $\mathbb{P}_{X \sim p_{\theta}}[f(X) = Q_{\theta'}^q(f)] = 0$, then $Q_{\theta}^q(f) < Q_{\theta'}^q(f)$.
\end{lemma}
The above result gives insights into the design and study of theoretically sounded black-box global optimization algorithms, for either discrete or continuous optimization. Indeed, showing that consecutive proposals achieve the increase condition \eqref{eq:strictIncreaseJ} allows to apply Lemma \ref{lemma:improvementJ_impliesQuantileImprovement} yielding a proposal with more mass where the objective function is low.

\subsection{The information-geometric optimization algorithm}

We now recall here the information-geometric (IGO) framework from \cite{ollivier2017, akimoto2013}. The latter is an iterative proposal construction algorithm, explicitly designed to achieve the increase condition \eqref{eq:strictIncreaseJ} at every iteration. The IGO framework has been shown in \cite{ollivier2017} to recover many existing algorithms to solve Problem \eqref{pblm:generalOpt}, both in discrete or continuous domains. Among the algorithms recovered by the IGO framework, let us mention the CE algorithm of \cite{kroese2006}.

The quantity $J(\theta | \theta')$ defined in \eqref{eq:defJ} is generally an intractable integral that needs in practice to be approximated with sampling. Throughout this paper, we focus on idealized algorithms that are deterministic, corresponding to the limit of an infinite number of samples. In such idealized setting, we only consider discrete-time updates since they are closer to a practical implementation than continuous flows. Two types of updates have been proposed in \cite{akimoto2013, ollivier2017} to satisfy the increase condition \eqref{eq:strictIncreaseJ}, leading to two distinct IGO-like algorithms, that we will recall here.

The first algorithm in \cite{ollivier2017, akimoto2013} is based on natural gradient ascent updates. Consider an iteration $k \in \mathbb{N}$, with $\theta_k$ parametrizing the current proposal, and the objective function being $J(\cdot | \theta_k)$. The natural gradient of $J(\cdot | \theta_k)$ at $\theta$ is the quantity $\widetilde{\nabla}_{\theta} J(\theta | \theta_k) = I(\theta)^{-1} \nabla_{\theta} J(\theta  | \theta_k)$, where $I(\theta) = -\mathbb{E}_{X \sim p_{\theta}}[ \nabla^2_{\theta} \ln p_{\theta}(X)]$ is the Fisher information matrix of $p_{\theta}$. Given the above gradient expression, iterating a gradient ascent scheme over $k \in \mathbb{N}$ leads to Algorithm \ref{alg:IGO_naturalGrad}. We remark that natural gradient updates have been used in other contexts than IGO, see for instance \cite{wierstra2014}.

\begin{algorithm}[H]
\caption{IGO algorithm (natural gradient update)}\label{alg:IGO_naturalGrad}
\begin{algorithmic}
\STATE Initialize $\theta_0$ and choose the step size $\tau > 0$.
\FOR{$k=0,\dots$}
    \STATE Update $\theta_{k+1}$ such that
    \begin{equation}
        \label{eq:IGO_naturalGradient}
        \theta_{k+1} = \theta_k + \tau \widetilde{\nabla}_{\theta} J(\theta | \theta_k)_{|\theta = \theta_k}.
     \end{equation}
\ENDFOR
\end{algorithmic}
\end{algorithm}

The second algorithm proposed in \cite{ollivier2017, akimoto2013} estimates the proposal parameters by performing a weighted maximum likelihood update at every iteration. We provide its description in Algorithm \ref{alg:IGO-ML}. The CE algorithm of \cite{kroese2006} is recovered as a special case when the step size is $\tau = 1$ and the weighting function is $w(u) = \delta_{u \leq q}(u)$~\cite{ollivier2017}.

\begin{algorithm}[H]
\caption{IGO algorithm (IGO-ML update)}\label{alg:IGO-ML}
\begin{algorithmic}
\STATE Initialize $\theta_0$ and choose the step size $\tau > 0$.
\FOR{$k=0,\dots$}
    \STATE Update $\theta_{k+1}$ such that 
    \begin{equation}
        \label{eq:IGO-ML}
        \theta_{k+1} = \argmax_{\theta \in \Theta} \left( (1-\tau) \int \ln (p_{\theta}(x))p_{\theta_k}(x)m(dx) + \tau \int W_{\theta_k}^f(x) \ln p_{\theta}(x) p_{\theta_k}(x)m(dx) \right).
    \end{equation}
\ENDFOR
\end{algorithmic}
\end{algorithm}

The theoretical properties of Algorithms \ref{alg:IGO_naturalGrad} and \ref{alg:IGO-ML} have been studied in \cite{ollivier2017, akimoto2013}. Algorithm \ref{alg:IGO-ML} achieves the increase condition \eqref{eq:strictIncreaseJ}, that is $J(\theta_{k+1} | \theta_k) > Z_w$ at every iteration $k\in \mathbb{N}$, for step sizes $\tau \in (0,1]$ \cite[Theorem 6]{akimoto2013}. This result gives in turn improvement guarantees for the CE algorithm thanks to Lemma \ref{lemma:improvementJ_impliesQuantileImprovement}. Algorithm \ref{alg:IGO_naturalGrad} has been shown to satisfy the increase condition \eqref{eq:strictIncreaseJ} for sufficiently small step sizes \cite[Proposition 7]{ollivier2017}. Moreover, Algorithms \ref{alg:IGO_naturalGrad} and \ref{alg:IGO-ML} have been shown to coincide when $\{ p_{\theta},\,\theta \in \Theta \}$ forms an exponential family \cite{barndorff-nielsen2014}, ensuring that Algorithm \ref{alg:IGO_naturalGrad} satisfies \eqref{eq:strictIncreaseJ} for $\tau \in (0,1]$ in this case (see \cite[Corollary 7]{akimoto2013}).

Algorithms \ref{alg:IGO_naturalGrad} and \ref{alg:IGO-ML} coincide with many existing black-box global optimization algorithms on discrete or continuous domains \cite[Section 5]{ollivier2017}, allowing to get improvement guarantees for these algorithms as well. However, the aforementioned study requires, as a preliminary step, to show that the considered algorithms fit within the IGO framework, which is not always possible nor straightforward. In the next section, we present our contribution, that aims at giving novel broader conditions under which the increase condition \eqref{eq:strictIncreaseJ} is satisfied. This allows, in particular, to exhibit new improvement guarantees beyond the IGO framework, the latter being retrieved as a special case.

\section{A general divergence-based condition for quantile improvement}
\label{section:mainResult}

We present our main results in this section. We start in Section \ref{sec:31} with the introduction of novel, divergence-based, conditions. We show that they imply the increase condition \eqref{eq:strictIncreaseJ}. We go further by quantifying the improvements in terms of reformulated objective and quantile. We then show in Section \ref{sec:32} that IGO algorithms satisfy our conditions, allowing us to provide a new and refined perspective on these methods. We finally exploit our divergence-based conditions to show new improvement guarantees for algorithms using proposals that do not form an exponential family. Namely, in Section \ref{sec:33}, we study a mixture-based algorithm and discuss his tight links with the mixture-based CE algorithm of \cite[Example 3.2]{kroese2006}. In Section \ref{sec:34}, we study an algorithm with heavy-tailed proposals, namely Student distributions with arbitrary degree of freedom parameter.

\subsection{Quantile improvement with divergence-decreasing steps}
\label{sec:31}

The goal of this section is to show that the increase condition \eqref{eq:strictIncreaseJ}, i.e., the theoretical guarantee achieved in the IGO framework, can be expressed as a consequence of divergence-based conditions, that we detail below. Combining these conditions with Lemma \ref{lemma:improvementJ_impliesQuantileImprovement} then yields a quantile improvement result. We also go further and quantify the improvement on the reformulated objective and quantile that come from our divergence-based conditions.

\subsubsection{Proposed divergence-based condition}
Our divergence-based conditions can be interpreted as the search for a proposal closer to a specific target distribution than the previous proposal, in the sense of the Kullback-Leibler or Rényi divergence. Let us start by specifying the target probability distribution we are going to consider. For $\theta \in \Theta$, we introduce $\pi_{\theta}^f$, the probability density with respect to $m$ defined for any $x \in \mathbb{X}$ by
\begin{equation}
    \label{eq:targetConstruction}
    \pi_{\theta}^f(x) = \frac{1}{Z_w} W_{\theta}^f(x) p_{\theta}(x).
\end{equation}
When $w(u) = \delta_{u \leq q}(u)$ for some $q \in (0,1)$, $\pi_{\theta}^f$ is a truncated version of $p_{\theta}$ with support being the points $x \in \mathbb{X}$ such that $q_{\theta}^<(x) < q$, meaning that areas of $\mathbb{X}$ where the values reached by $f$ are too high are given zero mass. Let a given iteration $k\in\mathbb{N}$. We aim at measuring the discrepancy between the target $\pi_{\theta_k}^f$ and either the current proposal, or the next one. This discrepancy is measured using the KL or a Rényi divergence, with $\alpha \in (0,1)\cup (1,+\infty)$. These are defined, respectively, for any probability densities $p_1, p_2$ with respect to $m$ by
\begin{align*}
    KL(p_1,p_2) &= \int \ln \left( \frac{p_1(x)}{p_2(x)}\right)p_1(x)m(dx),\\
    D_{\alpha}(p_1,p_2) &= \frac{1}{\alpha-1} \ln \left(\int p_1(x)^{\alpha} p_2(x)^{1-\alpha}m(dx) \right).
\end{align*}
If for some $x \in \mathbb{X}$, $p_1(x) = 0$, then we use $\ln(p_1(x))p_1(x) = 0$ (see \cite[Definition 7.1]{polyanskiy2023} for more details on these singular cases). Note that we have $D_{\alpha}(p_1,p_2) \xrightarrow[\alpha \rightarrow 1]{} KL(p_1,p_2)$ \cite[Theorem 5]{vanErven2014}, so that the KL divergence can be viewed as a limit case of the Rényi divergence.

We are now ready to state our first result, obtained when using the KL divergence to measure the discrepancy between probability densities. We show hereafter that, when a generic algorithm constructs its next proposal so that it gets closer, by a certain amount, to the target $\pi_{\theta_k}^f$ defined in \eqref{eq:targetConstruction}, then it improves upon the reformulated objective defined in \eqref{eq:defJ}, by an amount that we quantify.

\begin{proposition}
    \label{prop:KLDecreaseImpliesIntDecrease}
    Let $k \in \mathbb{N}$ and $\theta_k \in \Theta$. Suppose that $\pi_{\theta_k}^f$ is given by Equation \eqref{eq:targetConstruction}, and $p_{\theta_{k+1}}$ satisfies, for some $\Delta_k \in \mathbb{R}$,
    \begin{equation}
        \label{eq:KLStrictDecrease}
        KL(\pi_{\theta_k}^f, p_{\theta_{k+1}}) + \Delta_k \leq  KL(\pi_{\theta_k}^f, p_{\theta_k}).
    \end{equation}
    Then, $J(\theta_{k+1} | \theta_k) \geq \exp(\Delta_k) J(\theta_k | \theta_k) = \exp(\Delta_k) Z_w$. In particular, $(\theta_{k+1}, \theta_k)$ satisfy the increase condition \eqref{eq:strictIncreaseJ}, i.e., $J(\theta_{k+1} | \theta_k) > Z_w$ with $J$ defined in \eqref{eq:defJ}, if $\Delta_k > 0$.
\end{proposition}

\begin{proof}
    By construction of $\pi_{\theta_k}^f$, we can rewrite condition \eqref{eq:KLStrictDecrease} as
    \begin{equation*}
        \int \ln \left( \frac{W_{\theta_k}^f(x) p_{\theta_k}(x)}{Z_w p_{\theta_{k+1}(x)}} \right) \pi_{\theta_k}^f(x) m(dx) + \Delta_k\\ \leq \int \ln \left( \frac{W_{\theta_k}^f(x)}{Z_w} \right) \pi_{\theta_k}^f(x) m(dx),
    \end{equation*}
    and remark that it is equivalent to having 
    \begin{equation*}
        - \int \ln \left( \frac{p_{\theta_{k+1}(x)}}{p_{\theta_k}(x)} \right) \pi_{\theta_k}^f(x)m(dx) \leq -\Delta_k.
    \end{equation*}
    We then get from Jensen's inequality that 
    \begin{equation*}
        - \ln \left( \int \frac{p_{\theta_{k+1}}(x)}{p_{\theta_k}(x)}  \pi_{\theta_k}^f(x) m(dx) \right) \\ \leq - \int \ln \left( \frac{p_{\theta_{k+1}(x)}}{p_{\theta_k}(x)} \right) \pi_{\theta_k}^f(x)m(dx),
    \end{equation*}
    implying that 
    \begin{equation*}
        \int \frac{p_{\theta_{k+1}(x)}}{p_{\theta_k}(x)} \pi_{\theta_k}^f(x)m(dx) \geq \exp(\Delta_k).
    \end{equation*}
    Since by definition, $\frac{\pi_{\theta_k}^f(x)}{p_{\theta_k}(x)} = \frac{W_{\theta_k}^f(x)}{Z_w}$ for any $x \in \mathbb{X}$, it comes that $J(\theta_{k+1} | \theta_k) \geq \exp(\Delta_k) Z_w$.
\end{proof}

We now state a second similar result, that arises when one now measures the discrepancy between the target density and the proposals using a Rényi divergence.

\begin{proposition}
    \label{prop:alphaDivDecreaseImpliesIntDecrease}
    Let $k \in \mathbb{N}$, $\theta_k \in \Theta$ and suppose that $W_{\theta_k}^f$ takes values in $\{0,1\}$. Suppose that $\pi_{\theta_k}^f$ is given by Equation \eqref{eq:targetConstruction} and that $p_{\theta_{k+1}}$ satisfies, for some $\Delta_k \in \mathbb{R}$,
    \begin{equation}
        \label{eq:alphaDivStrictDecrease}
        D_{\alpha}(\pi_{\theta_k}^f, p_{\theta_{k+1}}) + \Delta_k \leq D_{\alpha}(\pi_{\theta_k}^f, p_{\theta_k}),
    \end{equation}
    for some $\alpha \in (0,1)$. Then, $J(\theta_{k+1} | \theta_k) \geq \exp(\Delta_k) J(\theta_k | \theta_k) = \exp(\Delta_k) Z_w$. In particular, $(\theta_{k+1}, \theta_k)$ satisfy the increase condition \eqref{eq:strictIncreaseJ} if $\Delta_k > 0$.
\end{proposition}

\begin{proof}
    By definition of $\pi_{\theta_k}^f$ and since by assumption, $W_{\theta_k}^f(x)^{\alpha} = W_{\theta_k}^f(x)$ for any $x \in \mathbb{X}$, Equation \eqref{eq:alphaDivStrictDecrease} is equivalent to
    \begin{align*}
        \frac{1}{\alpha - 1} \ln \left(\int \left( \frac{p_{\theta_{k+1}}(x)}{p_{\theta_k}(x)} \right)^{1-\alpha} W_{\theta_k}^f(x) p_{\theta_k}(x) m(dx) \right) + \Delta_k  &\leq \frac{1}{\alpha-1} \ln \left(\int W_{\theta_k}^f(x) p_{\theta_k}(x)m(dx) \right)\\ 
        &= \frac{1}{\alpha-1} \ln Z_w.
    \end{align*}
    We deduce from there that
    \begin{equation*}
        \int \left( \frac{p_{\theta_{k+1}}(x)}{p_{\theta_k}(x)} \right)^{1-\alpha} \pi_{\theta_k}^f(x) m(dx) \geq \exp((1-\alpha)\Delta_k).
    \end{equation*}
    Finally, since $u \longmapsto u^{1-\alpha}$ is concave due to the assumption on $\alpha$, we apply Jensen's inequality to obtain that
    \begin{equation*}
        \int \frac{p_{\theta_{k+1}}(x)}{p_{\theta_k}(x)} \pi_{\theta_k}^f(x) m(dx) \geq \exp(\Delta_k),
    \end{equation*}
    showing by definition of $\pi_{\theta_k}^f$, that $\int W_{\theta_k}^f(x) p_{\theta_{k+1}}(x) m(dx) \geq \exp(\Delta_k) Z_w$, and hence establishing the result.
\end{proof}

Propositions \ref{prop:KLDecreaseImpliesIntDecrease} and \ref{prop:alphaDivDecreaseImpliesIntDecrease} establish divergence-decrease conditions under which the increase condition \eqref{eq:strictIncreaseJ} is satisfied. Note that the construction mechanism of $p_{\theta_{k+1}}$ does not intervene in our result, while in \cite{ollivier2017, akimoto2013}, the increase condition \eqref{eq:strictIncreaseJ} was achieved for specific algorithms only.

\begin{remark}
    \label{remark:otherWeighting}
    Results equivalent to Propositions \ref{prop:KLDecreaseImpliesIntDecrease} and \ref{prop:alphaDivDecreaseImpliesIntDecrease} can be stated for other expectation-based reformulations of Problem \eqref{pblm:generalOpt} (under very mild hypotheses). For instance, consider the minimization over $\Theta$ of $J:\theta \longmapsto \mathbb{E}_{X \sim p_{\theta}}[\phi(f(X))]$ for some transform $\phi:\mathbb{R} \rightarrow \mathbb{R}_+$ (see for instance \cite{wierstra2008}). Consider as well the tilted densities $\pi^f_{\theta}$ defined for $\theta \in \Theta$ by $\pi^f_{\theta}(x) \propto \phi(f(x))p_{\theta}(x)$ for any $x \in \mathbb{X}$. Being able to define the probability density $\pi_{\theta}^f$ in this way is the only requirement on $\phi$. Then, decrease conditions of the form \eqref{eq:KLStrictDecrease}, or \eqref{eq:alphaDivStrictDecrease} if $\phi$ takes values in $\{0,1\}$, imply that
    \begin{equation*}
        \mathbb{E}_{X \sim p_{\theta_{k+1}}}[\phi(f(X))] \geq \exp(\Delta_k) \mathbb{E}_{X \sim p_{\theta_k}}[\phi(f(X))],
    \end{equation*}
    translating the improvement in terms of divergence to an improvement in the reformulation of Problem \eqref{pblm:generalOpt}.
\end{remark}

\subsubsection{Quantile improvement quantification}
We now present an additional result that quantifies how an improvement in term of the reformulation $J( \theta | \theta_k)$, defined in \eqref{eq:defJ}, results in an improvement in terms of $Q_{\theta}^q(f)$, defined in \eqref{eq:defQ_theta}. As we explained in the previous Section, in the particular case when $w(u) = \delta_{u \leq q}(u)$, there is a link between $J( \theta | \theta_k)$ and $Q_{\theta}^q(f)$. Our result can thus be seen as a quantitative and more precise version of Lemma \ref{lemma:improvementJ_impliesQuantileImprovement}. Although the previous results hold in the continuous and discrete settings, this result hereafter requires assumptions on $f$ and $\mathbb{X}$ that will hold for most continuous optimization problems but will usually not hold for discrete problems.

\begin{proposition}
    \label{prop:improvementJ_QuantileImprovementQantitative}
    Assume that $w(u) = \delta_{u \leq q}(u)$ for some $q \in (0,1)$. Let $k \in \mathbb{N}$ and $(\theta_k, \theta_{k+1}) \in \Theta^2$ such that $J(\theta_{k+1} | \theta_k ) \geq \exp(\Delta_k) Z_w$ with $\exp(\Delta_k) \mathbb{P}_{X \sim p_{\theta_{k+1}}}[f(X) \leq Q_{\theta_{k+1}}^q(f)] \in [0,1]$. Suppose that $F_{\theta_{k+1}}: u \longmapsto \mathbb{P}_{X \sim p_{\theta_{k+1}}}[f(X) \leq u]$ is continuous and strictly monotonic, and hence bijective from the range of $f$ (which is thus an interval of $\mathbb{R}$) to $[0,1]$. Then, we have
    \begin{equation}
        \label{eq:quantileChangeQuantitative}
        Q_{\theta_k}^q(f) \geq F_{\theta_{k+1}}^{-1}\left(\exp(\Delta_k) F_{\theta_{k+1}}(Q_{\theta_{k+1}}^q(f))\right).
    \end{equation}
    If $F_{\theta_{k+1}}$ is bijective and $\mathcal{C}^1$ in a neighbourhood of $Q_{\theta_{k+1}}^q(f)$ with $F_{\theta_{k+1}}'(Q_{\theta_{k+1}}^q(f)) \neq 0$, then
    \begin{equation}
        \label{eq:quantileChangeLinearized}
        Q_{\theta_k}^q(f) \geq Q_{\theta_{k+1}}^q(f) + q \Delta_k \left(F_{\theta_{k+1}}^{-1}\right)'(q) + o(\Delta_k^2)
    \end{equation}
    for $\Delta_k$ small enough, with $\left(F_{\theta_{k+1}}^{-1}\right)'(q) > 0$.
\end{proposition}

\begin{proof}
    We first show that $F_{\theta_{k+1}}(Q_{\theta_k}^q(f)) \geq q \Delta_k$, by adapting parts of the proof of \cite[Lemma 8]{akimoto2013}. First, following the arguments laid in the proof of \cite[Lemma 8]{akimoto2013}, we have that $W_{\theta_k}^f(x) = 0$ for any $x \in \mathbb{X}$ such that $f(x) > Q_{\theta_k}^q(f)$. This implies that $J(\theta_{k+1} | \theta_k) \leq \mathbb{P}_{X \sim p_{\theta_{k+1}}}[f(X) \leq Q_{\theta_k}^q(f)] = F_{\theta_{k+1}}(Q_{\theta_k}^q(f))$.

    From our assumptions, we thus have that $\exp(\Delta_k) q \leq F_{\theta_{k+1}}(Q_{\theta_k}^q(f))$ since our choice of $w$ implies $Z_w = q$. On the other hand, our assumption on $F_{\theta_{k+1}}$ implies that $Q_{\theta_{k+1}}^q(f)$ is the only value satisfying $q = F_{\theta_{k+1}}(Q_{\theta_{k+1}}^q(f))$. These two facts imply that $F_{\theta_{k+1}}(Q_{\theta_k}^q(f)) \geq \exp(\Delta_k) F_{\theta_{k+1}}(Q_{\theta_{k+1}}^q(f))$, establishing Equation \eqref{eq:quantileChangeQuantitative} with the bijectivity of $F_{\theta_{k+1}}$.

    In order to prove Equation \eqref{eq:quantileChangeLinearized}, we use Equation \eqref{eq:quantileChangeQuantitative} and develop $F_{\theta_{k+1}}^{-1} \circ \exp$ around $\ln F_{\theta_{k+1}}(Q_{\theta_{k+1}}^q(f))$. We get the well-posedness and sign of $\left(F_{\theta_{k+1}}^{-1}\right)'$ from the inverse function theorem.
\end{proof}

Proposition \ref{prop:improvementJ_QuantileImprovementQantitative} shows a complex interplay between the proposals and the objective $f$ through the cumulative density function $F_{\theta_{k+1}}$. This Proposition may be used to assess how a certain family of proposals is adapted to the problem at hand. Indeed, one could aim for a family of proposals such that $\left(F_{\theta_{\theta_{k+1}}}^{-1} \right)'(q)$ is as high as possible to ensure the largest quantile improvement possible.

In the particular case when $w(u) = \delta_{u \leq q}(u)$, we can then extend the result of Propositions \ref{prop:KLDecreaseImpliesIntDecrease} and \ref{prop:alphaDivDecreaseImpliesIntDecrease} with Lemma \ref{lemma:improvementJ_impliesQuantileImprovement} and Proposition \ref{prop:improvementJ_QuantileImprovementQantitative} to obtain the following quantile improvement results from divergence-decrease conditions.

\begin{corollary}
    \label{corollary:quantileImprovement}
    Assume that $w(u) = \delta_{u \leq q}(u)$ for some $q \in (0,1)$ and that, at a given iteration $k \in \mathbb{N}$, $p_{\theta_{k+1}}$ is constructed such that
    \begin{equation}
        \label{eq:alphaDivDecrease}
        D_{\alpha}(p_{\theta_{k+1}}, \pi_{\theta_{k}}^f) + \Delta_k \leq D_{\alpha}(p_{\theta_k}, \pi_{\theta_k}^f)
    \end{equation}
    for some $\Delta_k \in \mathbb{R}$ and $\alpha \in (0,1]$, with $\alpha = 1$ corresponding to the inequality $KL(\pi_{\theta_k}^f, p_{\theta_{k+1}}) + \Delta_k \leq  KL(\pi_{\theta_k}^f, p_{\theta_k})$.
    \begin{itemize}
        \item[$(i)$] Suppose that $\Delta_k > 0$. If $\alpha = 1$, or if $\alpha \in (0,1)$ and $W_{\theta_k}^f$ takes values in $\{0,1\}$, then $Q_{\theta_{k+1}}^q(f) \leq Q_{\theta_k}^q(f)$.
        \item[$(ii)$] If $\Delta_k > 0$ and $\mathbb{P}_{X\sim p_{\theta}} [f(X) = u] = 0$ for any $\theta \in \Theta$, $u \in \mathbb{R}$, then $Q_{\theta_{k+1}}^q(f) < Q_{\theta_k}^q(f)$.
        \item[$(iii)$] Suppose that $F_{\theta_{k+1}}: u \longmapsto \mathbb{P}_{X \sim p_{\theta}}[f(X) \leq u]$ is continuous and strictly monotonic. If $\alpha = 1$, or if $\alpha \in (0,1)$ and $W_{\theta_k}^f$ takes values in $\{0,1\}$, then Equation \eqref{eq:quantileChangeQuantitative} holds.
        \item[$(iv)$] Suppose that $F_{\theta_{k+1}}: u \longmapsto \mathbb{P}_{X \sim p_{\theta}}[f(X) \leq u]$ is bijective and $\mathcal{C}^1$ around $Q_{\theta_{k+1}}^q(f)$ with $F_{\theta_{k+1}}'(Q_{\theta_{k+1}}^q(f)) \neq 0$. If $\alpha = 1$, or if $\alpha \in (0,1)$ and $W_{\theta_k}^f$ takes values in $\{0,1\}$, then Equation \eqref{eq:quantileChangeLinearized} holds for $\Delta_k$ small enough.
    \end{itemize}
\end{corollary}

\begin{proof}
    Point $(i)$ follows from Propositions \ref{prop:KLDecreaseImpliesIntDecrease} and \ref{prop:alphaDivDecreaseImpliesIntDecrease} with the first part of Lemma \ref{lemma:improvementJ_impliesQuantileImprovement}. Point $(ii)$ follows by remarking that, under our assumptions, for any $\theta \in \Theta$, $x \in \mathbb{X}$, $q_{\theta}^<(x) = q_{\theta}^{\leq}(x)$, ensuring that $W_{\theta_k}^f(x)$ takes values in $\{0,1\}$. This also implies that $\mathbb{P}_{X \sim p_{\theta_{k+1}}}[f(X) = Q_{\theta_k}^q(f)] = 0$. The result comes by applying the second part of Lemma \ref{lemma:improvementJ_impliesQuantileImprovement}. Finally, points $(iii)$ and $(iv)$ are proven using the results from Propositions \ref{prop:KLDecreaseImpliesIntDecrease} and \ref{prop:alphaDivDecreaseImpliesIntDecrease} together with the results of Proposition \ref{prop:improvementJ_QuantileImprovementQantitative}.
\end{proof}

\subsubsection{Monitoring target construction}
We now show that the KL and Rényi divergences can also be used to control the discrepancy between the proposal $p_{\theta}$ and the resulting target $\pi_{\theta}^f$. The resulting bound only depends on the choice of the weighting function $w$.

\begin{proposition}
    \label{prop:targetControl}
    Consider $\theta \in \Theta$ and the probability densities $p_{\theta}$ and $\pi_{\theta}^f$. We have the following results, writing $D_1(\pi_{\theta}^f, p_{\theta})$ for $KL(\pi_{\theta}^f, p_{\theta})$.
    \begin{itemize}
        \item[$(i)$] If $W_{\theta}^f$ takes values in $\{0,1\}$, then $D_{\alpha}(\pi_{\theta}^f, p_{\theta}) = -\ln Z_w$ for any $\alpha > 0$.
        \item[$(ii)$] If $w$ takes values in $[0,1]$, then $D_{\alpha}(\pi_{\theta}^f, p_{\theta}) \leq -\ln Z_w$ for any $\alpha \in (0,1]$.
    \end{itemize}
\end{proposition}

\begin{proof}
    Consider any $\alpha \in (0,1) \cup (1,+\infty)$, we have
    \begin{equation}
        \label{eq:rewritingDalpha}
        D_{\alpha}(\pi_{\theta}^f, p_{\theta}) = \frac{1}{\alpha-1} \ln \left( \int \frac{W_{\theta}^f(x)^{\alpha}}{Z_w^{\alpha}} p_{\theta}(x)m(dx) \right).
    \end{equation}

    $(i)$ We have that $W_{\theta}^f(x)^{1-\alpha} = W_{\theta}^f(x)$ for any $x \in \mathbb{X}$, thus showing with Equation \eqref{eq:rewritingDalpha} that $D_{\alpha}(p_{\theta}, \pi_{\theta}^f) = - \ln Z_w$ for any $\alpha \in (0,1) \cup (1,+\infty)$ which implies the result, using \cite[Theorem 3]{vanErven2014} for the case of the KL divergence.

    $(ii)$ Since $w$ takes values in $[0,1]$, we also have $W_{\theta}^f(x) \in [0,1]$ for any $x \in \mathbb{X}$. Therefore, $W_{\theta}^f(x)^{\alpha} \geq W_{\theta}^f(x)$ for any $x \in \mathbb{X}$ when $\alpha \in (0,1)$. We thus get from Equation \eqref{eq:rewritingDalpha} when $\alpha \in (0,1)$ that $D_{\alpha}(\pi_{\theta}^f, p_{\theta}) \leq - \ln Z_w$. Taking the limit $\alpha \rightarrow 1$, $\alpha < 1$, we finally obtain that $KL(\pi_{\theta}^f, p_{\theta}) \leq - \ln Z_w$.
\end{proof}

\begin{remark}
    Consider, following Remark \ref{remark:otherWeighting}, that we use a target density of the form $\pi_{\theta}^f(x) \propto \phi(f(x)) p_{\theta}(x)$, using $\phi \circ f$ instead of $W_{\theta}^f$. It would be possible to derive results like Proposition \ref{prop:targetControl}. However, the normalization constant of $\pi_{\theta}^f$ is $\int \phi(f(x)) p_{\theta}(x) m(dx)$ which depends on $\theta$, while using $W_{\theta}^f$ ensures that the normalization constant of $\pi_{\theta}^f$ is equal to $Z_w$ for any $\theta \in \Theta$. 
\end{remark}

With Propositions \ref{prop:KLDecreaseImpliesIntDecrease} and \ref{prop:alphaDivDecreaseImpliesIntDecrease}, we have shown that if the divergence between the target and the next proposal is lower than the divergence between the target and the current proposal, Equation \eqref{eq:strictIncreaseJ} is satisfied. In the particular case of an indicator weighting function, Corollary \ref{corollary:quantileImprovement} shows that this leads to a quantile improvement. With Proposition \ref{prop:targetControl}, we have furthermore shown that the divergence between the target and the current proposal, from which the target is constructed, can be controlled by a quantity that depends only on the weighting function $w$. This means that, for any algorithm satisfying a divergence-decrease conditions at every step, divergences can also be used to understand both steps of the algorithms, namely the construction of the target, and the construction of the next proposal. We illustrate this fact in Figure \ref{fig:divergencePOV}.

\begin{figure*}[!ht]
\centering
\subfloat[]{\def\svgwidth{0.45\columnwidth}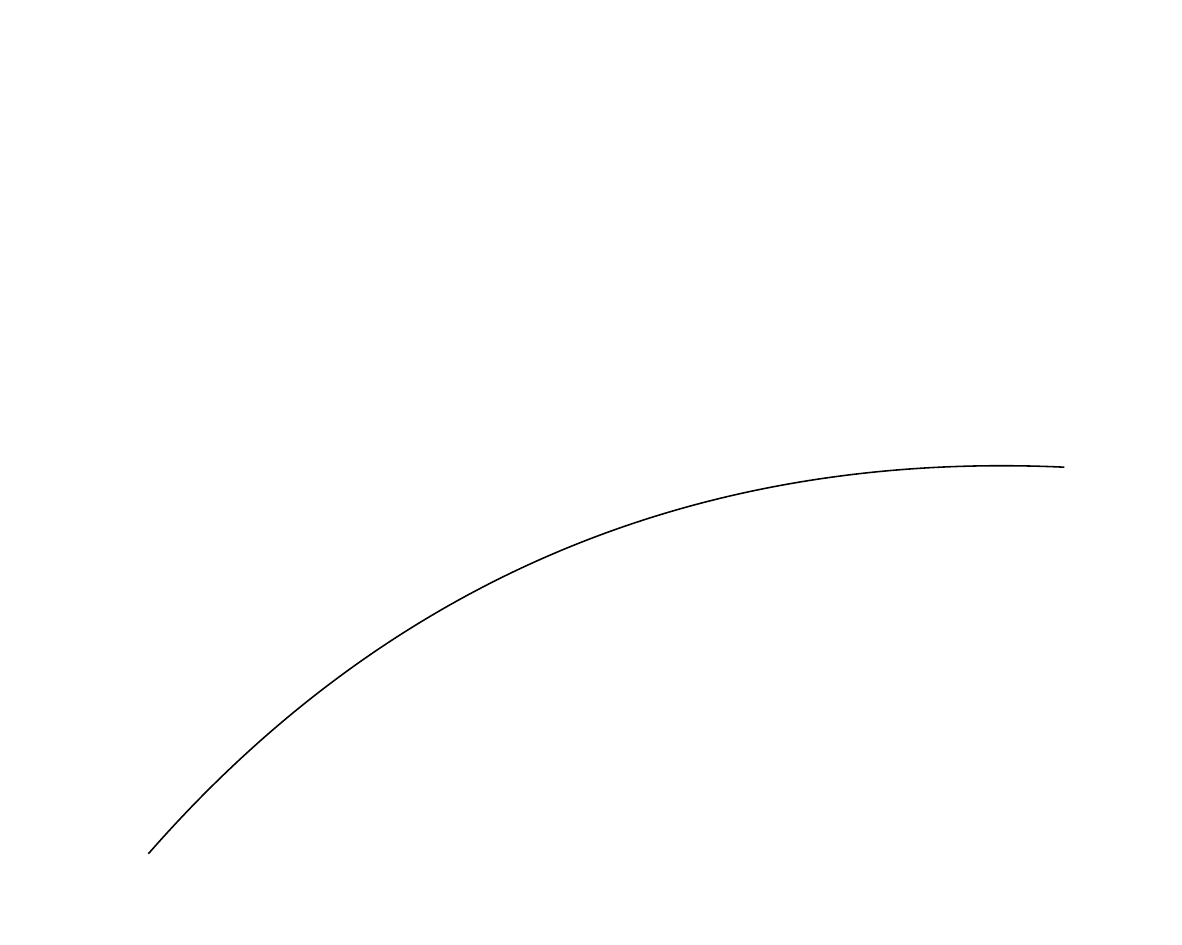}
\label{fig:divergenceAlpha}
\hfill
\subfloat[]{\def\svgwidth{0.45\columnwidth}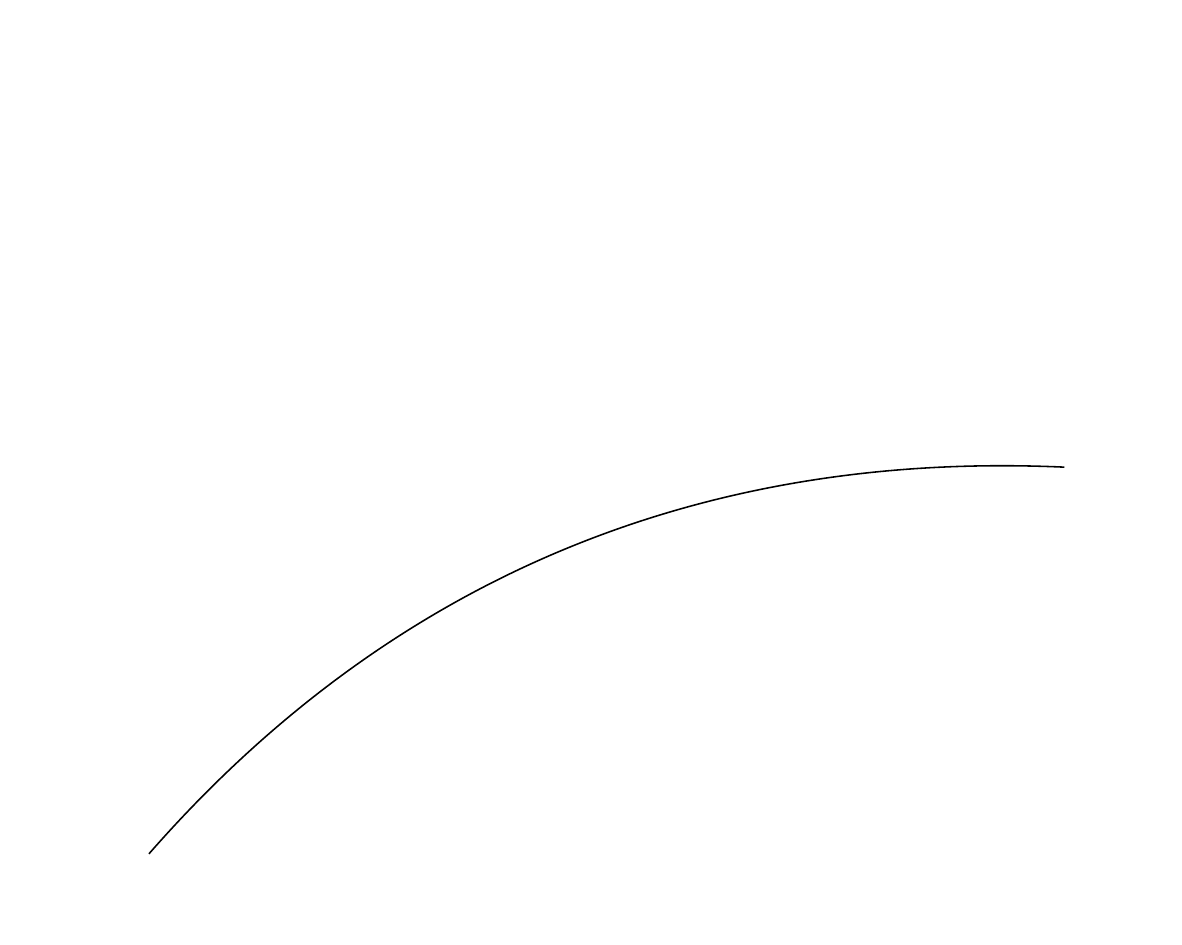}
\label{fig:divergenceKL}
\caption{A schematic view of one step of an algorithm covered by our divergence-based framework. Starting from a proposal $p_{\theta_k}$, one first constructs the target $\pi_{\theta_k}^f$ following \eqref{eq:targetConstruction}, i.e., a process that benefits from the results of Proposition \ref{prop:targetControl}, with $r = -\ln Z_w$. Then, one adapts $p_{\theta_{k+1}}$ such that a divergence-decrease \eqref{eq:alphaDivDecrease}, for some $\alpha \in (0,1]$, is achieved. (a) Case $\alpha \in (0,1]$ and $W_{\theta_k}^f$ taking values in $\{0,1\}$, corresponding to Propositions \ref{prop:KLDecreaseImpliesIntDecrease} or \ref{prop:alphaDivDecreaseImpliesIntDecrease} and Proposition \ref{prop:targetControl} $(i)$. (b) Case $\alpha \in (0,1]$ and $w$ taking values in $[0,1]$, corresponding to Propositions \ref{prop:KLDecreaseImpliesIntDecrease} or \ref{prop:alphaDivDecreaseImpliesIntDecrease} and Proposition \ref{prop:targetControl} $(ii)$.}
\label{fig:divergencePOV}
\end{figure*}

\subsubsection{Finite sample regime}
Our divergence-decrease conditions, presented in Propositions \ref{prop:KLDecreaseImpliesIntDecrease} and \ref{prop:alphaDivDecreaseImpliesIntDecrease}, are useful to analyze algorithms working in the finite-sample regime. Indeed, our results hold independently of the construction strategy adopted to define the next proposal. Actually, they even hold when the next proposal distribution degrades the performance upon the current one, which corresponds to the case $\Delta_k < 0$ in Equations \eqref{eq:KLStrictDecrease} and \eqref{eq:alphaDivStrictDecrease}. Such situation can typically arise in the finite sample regime, as we discuss hereafter.

In the case of finite sample size, the construction of the next proposal becomes inexact (i.e., noisy) as gradients need to be approximated. This makes the analysis of the algorithm more difficult. If the noise is controlled, in such a way that its effect is `absorbed' by the parameter $\Delta_k$ in the divergence-decrease conditions of Equations \eqref{eq:KLStrictDecrease} and \eqref{eq:alphaDivStrictDecrease}, then all of our results still apply (for instance in expectation or with high probability, depending on the control one has on the noise). Namely, it remains possible to establish improvement, or control the degradation if $\Delta_k < 0$, in terms of expectation-based reformulation and then in terms of quantile in such noisy contexts.

\subsection{Analyzing the IGO algorithms with our framework}
\label{sec:32}

\subsubsection{{Main result}} We now show that both IGO algorithms, namely Algorithms \ref{alg:IGO_naturalGrad} and \ref{alg:IGO-ML}, proposed in \cite{ollivier2017}, fall within our divergence-decrease framework, showing the applicability of our construction. We quantify the improvement achieved at each iteration in terms of divergence, which can then be used to quantify the quantile improvement using Proposition \ref{prop:improvementJ_QuantileImprovementQantitative}.

\begin{proposition}
    \label{prop:IGO_divDecrease}
    Consider a sequence $\{ \theta_k \}_{k \in \mathbb{N}}$ constructed either from Algorithm \ref{alg:IGO_naturalGrad} or Algorithm \ref{alg:IGO-ML}. Then, at every iteration $k \in \mathbb{N}$, we have the following.
    \begin{itemize}
        \item[$(i)$] If Algorithm \ref{alg:IGO_naturalGrad} is used, the proposal $p_{\theta_{k+1}}$ satisfies $KL(\pi_{\theta_k}^f, p_{\theta_{k+1}}) \leq KL(\pi_{\theta_k}^f, p_{\theta_k})$ for step sizes $\tau > 0$ small enough, with equality if and only if $\theta_{k+1} = \theta_k$. 
        \item[$(ii)$] If Algorithm \ref{alg:IGO_naturalGrad} is used and $\{p_{\theta},\,\theta \in \Theta\}$ is an exponential family, we have (under some assumptions detailed in the proof), that $KL(\pi_{\theta_k}^f, p_{\theta_{k+1}}) + \Delta_k \leq KL(\pi_{\theta_k}^f, p_{\theta_k})$ with $\Delta_k = \frac{1-\tau Z_w}{\tau Z_w} KL(p_{\theta_k}, p_{\theta_{k+1}})$.
        \item[$(iii)$] If Algorithm \ref{alg:IGO-ML} is used, then $KL(\pi_{\theta_k}^f,p_{\theta_{k+1}}) + \Delta_k \leq KL(\pi_{\theta_k}, p_{\theta_k})$ with $\Delta_k = \frac{1-\tau}{\tau Z_w} KL(p_{\theta_k}, p_{\theta_{k+1}})$.
    \end{itemize}
\end{proposition}

\begin{proof}
    Let $k \in \mathbb{N}$.
    
    $(i)$ The update \eqref{eq:IGO_naturalGradient} in Algorithm \ref{alg:IGO_naturalGrad} can be written as $\theta_{k+1} = \theta_k + \tau \int W_{\theta_k}^f(x) \widetilde{\nabla}_{\theta} \left( p_{\theta}(x) \right)_{|\theta = \theta_k} m(dx)$. Then, using $\widetilde{\nabla}_{\theta} \left( \ln p_{\theta}(x) \right)_{|\theta = \theta_k} = (1 / p_{\theta_k}(x)) \widetilde{\nabla}_{\theta} \left( p_{\theta}(x) \right)_{|\theta = \theta_k}$, we obtain that \eqref{eq:IGO_naturalGradient} is equivalent to having
    \begin{equation*}
        \theta_{k+1} = \theta_k + \tau \int W_{\theta_k}^f(x) p_{\theta_k}(x) \widetilde{\nabla}_{\theta} \left( \ln p_{\theta}(x) \right)_{|\theta = \theta_k} m(dx).
    \end{equation*}
    We can then notice that this is equivalent to performing $\theta_{k+1} = \theta_k - \tau Z_w \widetilde{\nabla}_{\theta} \left( KL(\pi_{\theta_k}^f, p_{\theta}) \right)_{|\theta = \theta_k}$, from which we deduce the result.

    $(ii)$ Suppose that $\{ p_{\theta}, \theta \in \Theta \}$ forms a minimal exponential family \cite{barndorff-nielsen2014} with sufficient statistics $\Gamma$ and log-partition function $A$ with $\Theta = \domain A$. Then, $A$ is differentiable on $\interior \domain A$ with $\nabla A(\theta) = \mathbb{E}_{X \sim p_{\theta}}[\Gamma(X)]$ \cite[Theorem 8.1]{barndorff-nielsen2014}. We also suppose that $(\theta_{k+1}, \theta_k) \in (\interior \Theta)^2$ and that for any $\theta \in \domain A$, $KL(\pi_{\theta_k}^f, p_{\theta}) < +\infty$. 

    From \cite[Equation (15)]{akimoto2013}, the IGO update over an exponential family at iteration $k$ reads
    \begin{equation*}
        \eta_{k+1} = \eta_k + \tau \int \left( W_{\theta_k}^f(x) (\Gamma(x) - \eta_k) \right)p_{\theta_k}(x)dx.
    \end{equation*}
    where $\eta_k = \nabla A(\theta_k)$ and $\eta_{k+1} = \nabla A(\theta_{k+1})$, both well-defined under our assumptions. This is equivalent to having 
    \begin{equation}
        \label{eq:optCond}
         \mathbb{E}_{X \sim \pi_{\theta_k}^f}[\Gamma(X)] - \nabla A(\theta_{k+1}) + \frac{1 - \tau Z_w}{\tau Z_w} (\nabla A(\theta_k)- \nabla A(\theta_{k+1})) = 0.
    \end{equation}
    We now interpret Equation \eqref{eq:optCond} as an optimality condition, showing that
    \begin{equation*}
        \theta_{k+1} = \argmin_{\theta \in \Theta} \left( KL(\pi_{\theta_k}^f, p_{\theta}) + \frac{1-\tau Z_w}{\tau Z_w} KL(p_{\theta_k}, p_{\theta}) \right).
    \end{equation*}
    This implies that $\theta_{k+1}$ is such that
    \begin{equation*}
        KL(\pi_{\theta_k}^f, p_{\theta_{k+1}}) + \frac{1-\tau Z_w}{\tau Z_w} KL(p_{\theta_k}, p_{\theta_{k+1}}) \leq KL(\pi_{\theta_k}^f, p_{\theta_k}).
    \end{equation*}

    We now show that Equation \eqref{eq:optCond} is the optimality condition of the problem solved by $\theta_{k+1}$. Consider any $\theta \in \Theta$ and a  probability density with respect to $m$, denoted by $p$. We have
    \begin{align*}
        KL(p, p_{\theta}) = \mathbb{E}_{X \sim p}[\ln p(X)] - \langle \mathbb{E}_{X \sim p}[\Gamma(X)], \theta\rangle + A(\theta),
    \end{align*}
    and thus obtain that $\nabla_{\theta} KL(p, p_{\theta}) = \nabla A(\theta) - \mathbb{E}_{X \sim p}[\Gamma(X)]$, which we then use to show the desired result.

    $(iii)$ The IGO-ML update \eqref{eq:IGO-ML} can be rewritten as
    \begin{align*}
        \theta_{k+1} &= \argmin_{\theta \in \Theta} \left( (1-\tau) \int \ln \left(\frac{1}{p_{\theta}(x)} \right) p_{\theta_k}(x)m(dx) + \tau \int \ln \left(\frac{1}{p_{\theta}(x)} \right) W_{\theta_k}^f(x) p_{\theta_k}(x)m(dx) \right)\\
        &= \argmin_{\theta \in \Theta} \left( (1-\tau) KL(p_{\theta_k}, p_{\theta}) + \tau Z_w KL(\pi_{\theta_k}, p_{\theta}) \right)\\
        &= \argmin_{\theta \in \Theta} \left( KL(\pi_{\theta_k}^f, p_{\theta}) + \frac{1-\tau}{\tau Z_w} KL(p_{\theta_k}, p_{\theta}) \right).
    \end{align*}
    We thus obtain by definition of $\theta_{k+1}$ that 
    \begin{equation*}
        KL(\pi_{\theta_k}^f, p_{\theta_{k+1}}) + \frac{1-\tau}{\tau Z_w} KL(p_{\theta_k},p_{\theta_{k+1}}) \leq KL(\pi_{\theta_k}^f, p_{\theta_k}).
    \end{equation*}
\end{proof}

\subsubsection{{Discussion}} Proposition \ref{prop:IGO_divDecrease} shows that the IGO algorithms can be studied using our divergence-decrease condition with the KL divergence. For Algorithm \ref{alg:IGO_naturalGrad}, i.e., the IGO algorithm using natural gradients, we recover in Proposition \ref{prop:IGO_divDecrease} $(i)$ that the increase \eqref{eq:strictIncreaseJ}, that is $J(\theta_{k+1} | \theta_k) > Z_w$, is guaranteed for infinitesimal step sizes \cite[Proposition 7]{ollivier2017}. Similarly, we recover in Proposition \ref{prop:IGO_divDecrease} $(iii)$ a similar result for Algorithm \ref{alg:IGO-ML}, i.e., the maximum-likelihood IGO algorithm, for any step size in $(0,1]$, which was established in \cite[Theorem 6]{akimoto2013}. 

Proposition \ref{prop:IGO_divDecrease} $(ii)$ establishes a similar result for Algorithm~\ref{alg:IGO_naturalGrad} when the proposals form an exponential family. In the proof of \cite[Corollary 7]{akimoto2013}, this result was achieved by remarking that in this case, Algorithms \ref{alg:IGO_naturalGrad} and \ref{alg:IGO-ML} coincide. We do a direct proof, showing that Algorithm \ref{alg:IGO_naturalGrad} is equivalent to a proximal update with a KL divergence objective. Note that \cite[Corollary 7]{akimoto2013} ensured improvement for step sizes in $(0,1]$ while our results allow for possibly larger step sizes (if $w(u) = \delta_{u \leq q}(u)$ with $q \in (0,1)$, $1/Z_w = 1/q > 1$). This is because the authors of \cite{akimoto2013} defined $W_{\theta}^f$ such that $Z_w = 1$, while we chose here a different convention.

\begin{remark}
    Proposition \ref{prop:IGO_divDecrease} $(ii)$ actually holds in the setting of Remark \ref{remark:otherWeighting}. More explicitly, we get the same result, with same $\Delta_k$, when optimizing $\theta \longmapsto \mathbb{E}_{X \sim p_{\theta}}[\phi(f(X))]$ over an exponential family using natural gradient descent with $\pi_{\theta}^f$ defined as in Remark \ref{remark:otherWeighting}. This allows to control the improvement of $\theta \longmapsto\mathbb{E}_{X \sim p_{\theta}}[\phi(f(X))]$ over iterations thanks to the result outlined in Remark \ref{remark:otherWeighting}.
\end{remark}

\subsection{A new result on mixture-based methods with our framework}

\label{sec:33}

We now show how our proposed framework can be applied for the study of black-box global optimization algorithms with mixture proposals. As already explained, such situation is challenging to analyze by sticking to the IGO framework, as mixture proposals do not form an exponential family nor yield a closed-form solution for the IGO-ML update \eqref{eq:IGO-ML} in Algorithm \ref{alg:IGO-ML}. We focus on a particular algorithm, linked both with the M-PMC algorithm of \cite{cappe2008}, a type of expectation-maximization (EM) algorithm proposed in the context of computational statistics, and with the mixture-based CE method proposed in \cite[Example 3.2]{kroese2006}. Similarly to IGO, our proposed algorithm can be applied to discrete and continuous optimization problems. By exploiting the paradigm we introduced in Section \ref{sec:31}, we show that each iteration of the considered algorithm achieves a divergence-decrease, thus implying that the increase condition \eqref{eq:strictIncreaseJ} is fulfilled. We can then apply Corollary \ref{corollary:quantileImprovement} to establish quantile improvement.

\subsubsection{Proposed algorithm} Let us first introduce the algorithm we are going to consider, summarized in Algorithm \ref{alg:mixtCE}. In this algorithm, the weight as well as the parameters of each component of the mixture are adapted at every iteration. We consider in the following mixture models with $J \in \mathbb{N}$ components $p_{\theta} = \sum_{j=1} \lambda^{(j)} p_{\vartheta^{(j)}}$ such that, for every $j \in \mathbb{N}$, $\lambda^{(j)} \in [0,1]$ and $p_{\vartheta^{(j)}} \in \{p_{\vartheta}, \, \vartheta \in \Theta\}$, with $\sum_{j=1}^J \lambda^{(j)} = 1$. We thus have the global parameters $\theta = (\{\lambda^{(j)}\}_{j=1}^J, \{ \vartheta^{(j)} \}_{j=1}^J)$.

\begin{algorithm}[h!]
\caption{Mixture-based ML algorithm}\label{alg:mixtCE}
\begin{algorithmic}
\STATE Initialize the parameters $\vartheta_0^{(j)}$ and the mixture weights $\lambda_0^{(j)}$ for $j=1,\dots, J$, and form the global parameter $\theta_0 = ( \{\lambda_0^{(j)}\}_{j=1}^J, \{\vartheta_0^{(j)}\}_{j=1}^J )$.
\FOR{$k=0,\dots$}
    \STATE For each $j = 1,\dots,J$, define the function $\rho_k^{(j)}:\mathbb{X} \rightarrow \mathbb{R}$ defined for any $x \in \mathbb{X}$ by
     \begin{equation}
         \rho_k^{(j)}(x) = \frac{\lambda_k^{(j)} p_{\vartheta_k^{(j)}}(x)}{\sum_{i=1}^J \lambda_k^{(i)} p_{\vartheta_k^{(i)}}(x)}.
     \end{equation}
    \STATE Update $\theta_{k+1} = ( \{\lambda_{k+1}^{(j)}\}_{j=1}^J, \{\vartheta_{k+1}^{(j)}\}_{j=1}^J )$ such that for every $j = 1,\dots,J$,
     \begin{align}
         \lambda_{k+1}^{(j)} &= \mathbb{E}_{X \sim \pi_{\theta_k}^f} [ \rho_k^{(j)}(X)],\label{eq:lambdaUpdateMixt}\\
         \vartheta_{k+1}^{(j)} &= \argmax_{\vartheta \in \Theta} \mathbb{E}_{X \sim \pi_{\theta_k}^f} [\ln p_{\vartheta}(X) \rho_k^{(j)}(X)].\label{eq:thetaUpdateMixt}
     \end{align}
\ENDFOR
\end{algorithmic}
\end{algorithm}

Algorithm \ref{alg:mixtCE} shares links with the EM point of view adopted in \cite{brookes2022}. In this work, several estimation-of-distribution algorithms \cite{larranaga2002} were shown to be EM algorithms with maximum likelihood steps that are reweighted using the objective to be minimized $f$ (see also the fitness EM algorithm of \cite{wierstra2008} and the discussion in \cite[Section 5.3]{akimoto2012}). Algorithm \ref{alg:mixtCE} recovers the M-PMC algorithm of \cite{cappe2008}, which is also an EM-like algorithm, but with rank-based weights (see \cite[Equation (14)]{ollivier2017} or Equation \eqref{eq:rankBasedWeights}) instead of importance weights. Note that, contrary to \cite{brookes2022} which does not explicitly consider mixture models, we do so here. Let us also remark that  Algorithm \ref{alg:IGO-ML} can also be linked to EM, using a similar analysis. 

\subsubsection{Main result}
We now show in Proposition \ref{prop:mixtKLdecrease} that Algorithm \ref{alg:mixtCE} achieves a decrease in terms of KL divergence at every iteration. Our proof techniques are reminiscent from the recent work \cite{daudel2023} on variational inference. The result of Proposition \ref{prop:mixtKLdecrease} can then be used to apply Proposition \ref{prop:KLDecreaseImpliesIntDecrease} and Corollary \ref{corollary:quantileImprovement} and get insights on the optimization performance of Algorithm \ref{alg:mixtCE}.

\begin{proposition}
    \label{prop:mixtKLdecrease}
    Consider a sequence $\{\theta_k\}_{k \in \mathbb{N}}$ generated by Algorithm \ref{alg:mixtCE} with $\theta_k = ( \{\lambda_k^{(j)}\}_{j=1}^J, \{\vartheta_k^{(j)}\}_{j=1}^J )$ for every $k\in \mathbb{N}$. Suppose that the problem in \eqref{eq:thetaUpdateMixt} is uniquely maximized at every iteration. Then, at every iteration $k \in \mathbb{N}$, Algorithm \ref{alg:mixtCE} achieves the decrease
    \begin{equation}
        \label{eq:KLdecreaseMixt}
        KL \left(\pi_{\theta_k}^f, p_{\theta_{k+1}} \right) + \Delta_k \leq KL \left(\pi_{\theta_k}^f, p_{\theta_k} \right),
    \end{equation}
    with $\Delta_k > 0$, unless $\lambda_{k+1}^{(j)} = \lambda_k^{(j)}$ and $\vartheta_{k+1}^{(j)} = \vartheta_k^{(j)}$ for every $j =1,\dots,J$, in which  case $\Delta_k = 0$.
\end{proposition}

\begin{proof}
We adapt the ideas of the proof of \cite[Theorem 2]{daudel2023}. We compute the quantity
\begin{align*}
    KL(\pi_{\theta_k}^f, p_{\theta_{k+1}}) - KL(\pi_{\theta_k}^f, p_{\theta_k})
    &= \int -\ln \left( \frac{\sum_{j=1}^J \lambda_{k+1}^{(j)} p_{\vartheta_{k+1}^{(j)}}(x)}{\sum_{i=1}^J \lambda_{k}^{(i)} p_{\vartheta_{k}^{(i)}}(x)} \right) \pi_{\theta_k}^f(x)m(dx)\\
    &= \int -\ln \left( \sum_{j=1}^J \rho_k^{(j)}(x) \frac{\lambda_{k+1}^{(j)} p_{\vartheta_{k+1}^{(j)}}(x)}{\lambda_k^{(j)} p_{\vartheta_{k}^{(j)}}(x)} \right) \pi_{\theta_k}^f(x) m(dx)\\
    &\leq \int  -\sum_{j=1}^J\rho_k^{(j)}(x) \ln \left(\frac{\lambda_{k+1}^{(j)} p_{\vartheta_{k+1}^{(j)}}(x)}{\lambda_k^{(j)} p_{\vartheta_{k}^{(j)}}(x)} \right) \pi_{\theta_k}^f(x) m(dx)
\end{align*}
using Jensen's inequality and that the $\rho_k^{(j)}(x)$ sum to one for any $x \in \mathbb{X}$. We can then decompose the above quantity into two terms, namely,
\begin{multline}
    \label{eq:KLdiffDecomp}
    \int  -\sum_{j=1}^J \rho_k^{(j)}(x) \ln \left(\frac{\lambda_{k+1}^{(j)} p_{\vartheta_{k+1}^{(j)}}(x)}{\lambda_k^{(j)} p_{\vartheta_{k}^{(j)}}(x)} \right) \pi_{\theta_k}^f(x) m(dx) = - \sum_{j=1}^J \ln \left( \frac{\lambda_{k+1}^{(j)}}{\lambda_k^{(j)}} \right) \int \rho_k^{(j)}(x)\pi_{\theta_k}^f(x)m(dx) \\
    + \sum_{j=1}^J \int \rho_k^{(j)}(x) \ln \left(\frac{p_{\vartheta_k^{(j)}}(x)}{p_{\vartheta_{k+1}^{(j)}}(x)}\right) \pi_{\theta_k}^f(x)m(dx).
\end{multline}
Due to the definition of $\lambda_{k+1}$, given in Equation \eqref{eq:lambdaUpdateMixt}, that is $\lambda_{k+1}^{(j)} = \int \rho_k^{(j)}(x) \pi_{\theta_{k}}^f(x)m(dx)$, the first term in the right-hand side of Equation \eqref{eq:KLdiffDecomp} is equal to $-\sum_{j=1}^J \ln \left( \frac{\lambda_{k+1}^{(j)}}{\lambda_k^{(j)}} \right) \lambda_{k+1}^{(j)}$ which is non-positive from Jensen's inequality, being null if and only if $\lambda_k = \lambda_{k+1}$. The second term in the right-hand side of \eqref{eq:KLdiffDecomp} is a sum of $J$ terms, each being non-positive from the definition of $\vartheta_{k+1}^{(j)}$ given in Equation \eqref{eq:thetaUpdateMixt}. Each term is zero if and only if $\vartheta_{k+1}^{(j)} = \vartheta_{k}^{(j)}$, due to our assumption that each maximization problem of the form \eqref{eq:thetaUpdateMixt} is uniquely maximized. We have thus shown the decrease \eqref{eq:KLdecreaseMixt}, with equality holding if and only if $\lambda_{k+1} = \lambda_k$ and $\theta_{k+1} = \theta_k$.
\end{proof}

\subsubsection{Discussion} Since our Proposition \ref{prop:mixtKLdecrease} can be used to apply Corollary~\ref{corollary:quantileImprovement}, it can be viewed, to our knowledge, as the first result to establish quantile improvement for black-box global optimization algorithms that are explicitly mixture-based. Indeed, mixtures were not explicitly considered in \cite{ollivier2017, akimoto2013}, and they often do not admit closed-form solutions for the maximization problem \eqref{eq:IGO-ML} in Algorithm \ref{alg:IGO-ML}. The strategy adopted in the literature was usually to perform EM-like updates, as it was done in \cite[Example 3.2]{kroese2006} for instance, which can now be handled with our divergence-decrease condition. Many variational inference or adaptive importance sampling methods explicitly consider mixtures, see for instance \cite{cappe2008,bugallo2017adaptive, daudel2023}, showing the potential for further links between black-box global optimization with mixture models and variational inference. Let us also remark that compared to more complex mixture-based algorithms, such as the ones proposed in \cite{maree2017, ahrari2017, ahrari2022}, whose convergence has only been verified empirically, the proposed Algorithm \ref{alg:mixtCE} has a fixed number of mixture components.

We now discuss the links between Algorithm \ref{alg:mixtCE} and the CE algorithm of \cite[Example 3.2]{kroese2006}. To make these links more explicit, we discuss the finite sample size implementation of Algorithm \ref{alg:mixtCE}. This requires the computation of integrals with respect to $\pi_{\theta_k}^f$, as discussed in \cite{ollivier2017}. To do so, $N$ points $x_{k,n}$, $n=1,\dots,N$, are first sampled from the mixture distribution $p_{\theta_k} = \sum_{j=1}^J \lambda_k^{(j)} q_{\vartheta_k^{(j)}}$. This is done by drawing a component $j$ with probability $\lambda_k^{(j)}$ via multinomial sampling, and then drawing from $p_{\vartheta_k^{(j)}}$. Each sample $x_{k,n}$ receives a rank-based weight $\widehat{\omega}_{k,n}$ defined as in \cite[Equation (14)]{ollivier2017} by
\begin{equation}
    \label{eq:rankBasedWeights}
    \widehat{\omega}_{k,n} = \frac{1}{N} w \left( \frac{\rk(x_{k,n}) + 1/2}{N} \right),
\end{equation}
where $\rk (x_{k,n})$ is the number of samples in $\{x_{k,n}\}_{n=1}^N$ with value of $f$ strictly less than $f(x_{k,n})$. Then, one can show using \cite[Proposition 27]{ollivier2017} and Slutsky's Lemma that for any integrand $h$ such that $\mathbb{E}_{X \sim p_{\theta_k}}[h(X)^2] < +\infty$, and conditioned on $\theta_k$, that
\begin{equation}
    \frac{1}{\sum_{n=1}^N \widehat{\omega}_{k,n}}\sum_{n=1}^N \widehat{\omega}_{k,n} h(x_{k,n}) \xrightarrow[N \rightarrow +\infty]{a.s.} \mathbb{E}_{X \sim \pi_{\theta_k}^f}[h(X)].
\end{equation}

In light of the above formula, Algorithm \ref{alg:mixtCE} can be approximated at iteration $k \in \mathbb{N}$ by setting, for every $j = 1,\dots,J$,
\begin{align*}
    \lambda_{k+1}^{(j)} &= \frac{1}{\sum_{n=1}^N \widehat{\omega}_{k,n}}\sum_{n=1}^N \widehat{\omega}_{k,n} \rho_k^{(j)}(x_{k,n}),\\
    \vartheta_{k+1}^{(j)} &= \argmax_{\vartheta \in \Theta} \frac{1}{\sum_{n=1}^N \widehat{\omega}_{k,n}}\sum_{n=1}^N \widehat{\omega}_{k,n} \ln p_{\vartheta}(x_{k,n}) \rho_k^{(j)}(x_{k,n}).
\end{align*}

In order to compare with the mixture-based CE algorithm of \cite[Example 3.2]{kroese2006}, let us introduce $\xi_{k,n}^{(j)}$ which is a latent variable being equal to $1$ if $x_{k,n}$ has been sampled from the component $j$ of the mixture and zero otherwise. Then, the CE algorithm of \cite[Example 3.2]{kroese2006} has the following update at iteration $k \in \mathbb{N}$ and for every $j = 1,\dots,J$.
\begin{align*}
    \lambda_{k+1}^{(j)} &= \frac{1}{\sum_{n=1}^N \widehat{\omega}_{k,n}}\sum_{n=1}^N \widehat{\omega}_{k,n} \xi_{k,n}^{(j)},\\
    \vartheta_{k+1}^{(j)} &= \argmax_{\vartheta \in \Theta} \frac{1}{\sum_{n=1}^N \widehat{\omega}_{k,n}}\sum_{n=1}^N \widehat{\omega}_{k,n} \ln p_{\vartheta}(x_{k,n}) \xi_{k,n}^{(j)}.
\end{align*}

We thus observe that the approximated version of Algorithm \ref{alg:mixtCE} and the mixture-based CE algorithm are very similar, except that $\xi_{k,n}^{(j)}$ is used instead of $\rho_k^{(j)}(x_{k,n})$ in the latter. Since $\rho_{k}^{(j)}(x_{k,n}) = \mathbb{E}[ \xi_{k,n}^{(j)} | x_{k,n}]$, using $\rho_{k}^{(j)}(x_{k,n})$ instead of $\xi_{k,n}$ amounts to a Rao-Blackwellized version (i.e., a random variable is replaced by its conditional expectation \cite{cappe2008}) of the CE algorithm from \cite[Example 3.2]{kroese2006}. The procedure used in the approximated version of Algorithm \ref{alg:mixtCE} does not entail additional evaluations of the objective $f$, while providing better numerical stability \cite{cappe2008}, as all the components of the mixtures appear in every update. This shows how our divergence-based conditions can be used to better understand and design algorithms for mixture-based proposals. It also allows to turn divergence-minimizing methods such as the M-PMC algorithm of \cite{cappe2008} into black-box global optimization algorithms, simply by changing the target distribution and using rank-based weights.

\subsection{A new result for heavy-tailed proposals with our framework}
\label{sec:34}

We finally apply our theoretical tools to study a black-box global optimization algorithm with proposals being Student distributions with a fixed degree of freedom parameter $\nu > 0$. Specifically, we propose an algorithm to update the location and scale parameters of the proposals at every iteration, and show that it satisfies our divergence-decrease conditions. Note that, contrary to our previous analysis which also held for discrete problems, in this section, we now assume that $\mathbb{X} = \mathbb{R}^d$ and take $m$ to be the Lebesgue measure.

\subsubsection{Proposed algorithm} We consider Student distributions in dimension $d$ with $\nu > 0$ degrees of freedom indexed by their location parameters $\mu \in \mathbb{R}^d$ and scale parameters $\Sigma \in \mathcal{S}_{++}^d$, the set of positive definite matrices in dimensions $d$. The density with respect to the Lebesgue measure of the Student distribution $\mathcal{T}(\cdot ; \mu, \Sigma, \nu)$ is defined for all $x \in \mathbb{R}^d$ by
\begin{equation}
    \label{eq:studentDef}
    \mathcal{T}(x ; \mu, \Sigma, \nu) \propto \left(1 + \frac{1}{\nu}(x-\mu)^{\top}\Sigma^{-1}(x-\mu) \right)^{- \frac{\nu+d}{2}}
\end{equation}
with normalization constant being equal to $\frac{\Gamma(\nu/2)}{\Gamma((\nu+d)/2)} (\nu^d \pi^d \det(\Sigma))^{1/2}$, $\Gamma$ denoting the Gamma function and $\det$ the determinant. When $\nu = 1$, the Cauchy distributions are recovered, while Gaussian distributions are recovered in the limit $\nu \rightarrow +\infty$. Alternatively, the density in \eqref{eq:studentDef} can be written as the continuous mixture
\begin{equation}
    \label{eq:studentDefGamma}
    \mathcal{T}(x ; \mu, \Sigma, \nu) = \int_0^{+\infty} \mathcal{N}\left(x ;\mu, \frac{1}{z}\Sigma\right) \mathcal{G}\left(z ; \frac{\nu}{2}, \frac{\nu}{2}\right) dz,
\end{equation}
where the latent variable $Z$ is distributed following the Gamma distribution with parameters $(\frac{\nu}{2}, \frac{\nu}{2})$ and probability density $\mathcal{G}(z ; \frac{\nu}{2}, \frac{\nu}{2})$ for any $z \in (0,+\infty)$. Conditionally on $Z$, $X$ follows a Gaussian distribution with mean $\mu$ and covariance $\frac{1}{Z}\Sigma$, and density $\mathcal{N}(x ; \mu, \frac{1}{Z}\Sigma)$ for any $x \in \mathbb{R}^d$. We will use the point of view from \eqref{eq:studentDefGamma} in the following. We fix $\nu > 0$, and consider parameters $ \theta = (\mu, \Sigma)$ with associated densities $p_{\theta} = \mathcal{T}(\cdot ; \mu, \Sigma, \nu)$. In this context, we propose the heavy-tailed black-box global optimization algorithm, summarized in Algorithm~\ref{alg:heavyCE}.

\begin{algorithm}[H]
\caption{Heavy-tail ML algorithm}\label{alg:heavyCE}
\begin{algorithmic}
\STATE Initialize the parameters $\theta_0 = ( \mu_0, \Sigma_0)$ and choose the degree of freedom parameter $\nu > 0$.
\FOR{$k=0,\dots$}
    \STATE Define the function $\gamma_k^{(\nu)}:\mathbb{X} \rightarrow \mathbb{R}$ defined for any $x \in \mathbb{X}$ by
     \begin{equation}
        \label{eq:defGammaNu}
         \gamma_k^{(\nu)}(x) = \frac{\nu + d}{\nu + (x-\mu_k)^{\top}\Sigma_k^{-1}(x-\mu_k)}.
     \end{equation}
    \STATE Update $\theta_{k+1} = ( \mu_{k+1}, \Sigma_{k+1} )$ such that
     \begin{align}
         \mu_{k+1} &= \frac{\mathbb{E}_{X \sim \pi_{\theta_k}^f}[\gamma_k^{(\nu)}(X)X]}{\mathbb{E}_{X \sim \pi_{\theta_k}^f}[\gamma_k^{(\nu)}(X)]} ,\label{eq:muUpdateHeavy}\\
         \Sigma_{k+1} &= \frac{\mathbb{E}_{X \sim \pi_{\theta_k}^f}[\gamma_k^{(\nu)}(X) X X^{\top}]}{\mathbb{E}_{X \sim \pi_{\theta_k}^f}[\gamma_k^{(\nu)}(X)]} - \mu_{k+1} \mu_{k+1}^{\top} .\label{eq:SigmaUpdateHeavy}
     \end{align}
\ENDFOR
\end{algorithmic}
\end{algorithm}

When the degree of freedom parameter $\nu$ goes to infinity, we have that $\mathcal{T}(x ; \mu, \Sigma, \nu) \rightarrow \mathcal{N}(x ; \mu, \Sigma)$ for any $x \in \mathbb{R}^d$, meaning that Student distributions recover the Gaussian distributions. Moreover, we have at any iteration $k \in \mathbb{N}$ that $\gamma_k^{(\nu)}(x) \rightarrow 1$ when $\nu \rightarrow +\infty$. In this case, the updates \eqref{eq:muUpdateHeavy} and \eqref{eq:SigmaUpdateHeavy} in Algorithm \ref{alg:heavyCE} recover the updates of Algorithm \ref{alg:IGO-ML} with step size $\tau = 1$ when Gaussian distributions are used. Moreover, evaluating the function $\gamma_k^{(\nu)}$ does not imply a heavy computational burden, as it does not involve additional computations of the objective function $f$.

\subsubsection{Main result} We now show that Algorithm \ref{alg:heavyCE} achieves our divergence-decrease condition. This means that the improvement \eqref{eq:strictIncreaseJ} is satisfied at every iteration, and thus that one can use Corollary~\ref{corollary:quantileImprovement} to get quantile improvement when $w(u) = \delta_{u \leq q}(u)$ is used.

\begin{proposition}
    \label{prop:improvementHeavy}
    Consider a sequence $\{\theta_k\}_{k \in \mathbb{N}}$ generated by Algorithm \ref{alg:heavyCE}. At every iteration $k \in \mathbb{N}$, we have the decrease
    \begin{equation}
        \label{eq:decreaseKLHeavy}
        KL(\pi_{\theta_k}^f, p_{\theta_{k+1}}) + \Delta_k \leq KL(\pi_{\theta_k}^f, p_{\theta_k}),
    \end{equation}
    with $\Delta_k > 0$, unless $\theta_{k+1} = \theta_k$ in which case $\Delta_k = 0$.
\end{proposition}

\begin{proof}
    Consider any $\theta = (\mu, \Sigma) \in \Theta = \mathbb{R}^d \times \mathcal{S}_{++}^d$ and any distribution $p$ over the optimization variables $x \in \mathbb{R}^d$ and the latent variables $z \in (0,+\infty)$. We then have
    \begin{align*}
        \int \ln p_{\theta}(x)\pi_{\theta_k}^f(x)dx &= \iint \ln p_{\theta}(x) p(z | x) dz\, \pi_{\theta_k}^f(x)dx\\
        &= \iint \ln \left( \frac{p_{\theta}(x,z)}{p_{\theta}(z | x)}\right) p(z | x) dz\, \pi_{\theta_k}^f(x)dx\\
        &= \iint \ln \left( \frac{p_{\theta}(x,z)}{p(z|x)} \right) p(z|x) dz\, \pi_{\theta_k}^f(x)dx - \iint \ln \left( \frac{p_{\theta}(z|x)}{p(z|x)} \right) p(z|x) dz\, \pi_{\theta_k}^f(x)dx\\
        &=\iint \ln \left( \frac{p_{\theta}(x,z)}{p(z|x)} \right) p(z|x) dz\, \pi_{\theta_k}^f(x)dx + \int KL(p(\cdot | x), p_{\theta}(\cdot | x)) \pi_{\theta_k}^f(x)dx.
    \end{align*}
    Hence, we have that
    \begin{equation}
        \label{eq:rewritingLlk}
        \int \ln p_{\theta}(x)\pi_{\theta_k}^f(x)dx\\ \geq \iint \ln \left( \frac{p_{\theta}(x,z)}{p(z|x)} \right) p(z|x) dz\, \pi_{\theta_k}^f(x)dx,
    \end{equation}
    with equality if and only if $p_{\theta}(z | x) = p(z | x)$ for any $z \in (0,+\infty)$ and $x \in \mathbb{R}^d$.
    
    We now compute the gap in Kullback-Leibler divergence and using Equation \eqref{eq:rewritingLlk}, we obtain
    \begin{align*}
        KL(\pi_{\theta_k}^f, p_{\theta_{k+1}}) - KL(\pi_{\theta_k}^f, p_{\theta_k})
        &= -\int \ln  p_{\theta_{k+1}}(x) \pi_{\theta_k}^f(x)dx + \int \ln  p_{\theta_{k}}(x) \pi_{\theta_k}^f(x)dx\\
        &\leq - \iint \ln \left( \frac{p_{\theta_{k+1}}(x,z)}{p_{\theta_k}(z|x)} \right) p_{\theta_k}(z|x) dz\, \pi_{\theta_k}^f(x)dx\\ &\qquad\qquad+ \iint \ln \left( \frac{p_{\theta_{k}}(x,z)}{p_{\theta_k}(z|x)} \right) p_{\theta_k}(z|x) dz\, \pi_{\theta_k}^f(x)dx\\
        &= - \iint \ln  p_{\theta_{k+1}}(x,z) p_{\theta_k}(z|x) dz\, \pi_{\theta_k}^f(x)dx\\ &\qquad\qquad+ \iint \ln p_{\theta_{k}}(x,z) p_{\theta_k}(z|x) dz\, \pi_{\theta_k}^f(x)dx.
    \end{align*}
    Since the degree of freedom parameter is kept constant, we have that $p_{\theta}(x,z) = p_{\theta}(x | z) p(z)$, with $p_{\theta}(x|z) = \mathcal{N}(x ; \mu, \frac{1}{z}\Sigma)$ and $p(z) = \Gamma(z ; \frac{\nu}{2}, \frac{\nu}{2})$ that does not depend on $\theta$. In particular, we can write
    \begin{multline*}
        KL(\pi_{\theta_k}^f, p_{\theta_{k+1}}) - KL(\pi_{\theta_k}^f, p_{\theta_k}) \leq - \iint \ln p_{\theta_{k+1}}(x|z) p_{\theta_k}(z|x) dz\, \pi_{\theta_k}^f(x)dx\\ + \iint \ln p_{\theta_{k}}(x|z) p_{\theta_k}(z|x) dz\, \pi_{\theta_k}^f(x)dx.
    \end{multline*}
    Therefore, showing that $\theta_{k+1} = (\mu_{k+1}, \Sigma_{k+1})$ is such that
    \begin{equation}
        \label{eq:KLdecrReducedToOptProblem}
        \theta_{k+1} = \argmax_{\theta \in \Theta } \iint \ln  p_{\theta}(x|z) p_{\theta_k}(z|x) dz\, \pi_{\theta_k}^f(x)dx,
    \end{equation}
    establishes the decrease in Equation \eqref{eq:decreaseKLHeavy} with equality if and only if $\theta_{k+1} = \theta_k$. We now show that $\theta_{k+1} = (\mu_{k+1}, \Sigma_{k+1})$ as constructed in Algorithm \ref{alg:heavyCE} satisfies \eqref{eq:KLdecrReducedToOptProblem}.

    For any $\theta \in \Theta$, we have that $p_{\theta}(x|z) = \mathcal{N}(x ; \mu, \frac{1}{z}\Sigma)$. Hence, we can compute that
    \begin{equation*}
        \iint \ln p_{\theta}(x|z) p_{\theta_k}(z|x)dz\, \pi_{\theta_k}^f(x)dx = - \ln \det(\Sigma) - \frac{1}{2} \iint z p_{\theta_k}(z|x)dz\, (x-\mu)^{\top}\Sigma^{-1}(x-\mu) \pi_{\theta_k}^f(x)dx.
    \end{equation*}
    For any $x \in \mathbb{R}^d$, one can check that $p_{\theta_k}( \cdot | x)$ is the density of a Gamma distribution with parameters $\left(\frac{\nu+d}{2}, \frac{1}{2}(\nu + (x-\mu_k)\Sigma_k^{-1}(x-\mu_k))\right)$. We then remark that $\gamma_k^{(\nu)}$ as defined in \eqref{eq:defGammaNu} satisfies for any $x \in \mathbb{R}^d$
    \begin{equation*}
        \gamma_k^{(\nu)}(x) = \int z p_{\theta_k}(z | x) dz
    \end{equation*}
    and we thus get an explicit expression for the objective in \eqref{eq:KLdecrReducedToOptProblem} of the form
    \begin{equation*}
        \int \ln p_{\theta}(x|z) p_{\theta_k}(z|x)dz\, \pi_{\theta_k}^f(x)m(dx) = - \ln \det (\Sigma) - \frac{1}{2} \mathbb{E}_{X \sim \pi_{\theta_k}^f} [\gamma_k^{(\nu)}(X)(X-\mu)^{\top}\Sigma^{-1}(X-\mu)],
    \end{equation*}
    from which the result follows.
\end{proof}

\subsubsection{{Discussion}} The result of our Proposition \ref{prop:improvementHeavy} allows to give improvement guarantees for heavy-tailed distributions, that do not form an exponential family. In particular, this applies for any Student family, including Cauchy distributions when $\nu =1$, and to Gaussian distributions in the limit $\nu \rightarrow +\infty$. Cauchy proposals perform better than Gaussian proposals in low dimension, while the reverse is true when the dimension of the problem grows \cite{sanyang2016}. Our algorithm allows to interpolate these two regimes, possibly opening the way to a tail-adaptive algorithm able to select good values of the degree of freedom parameter, as it is done for instance in \cite{daudel2023, guilmeau2024}. 

Similarly to the result of Proposition \ref{prop:mixtKLdecrease}, the result of Proposition \ref{prop:improvementHeavy} leverages an analysis used on existing divergence-minimization algorithms. This shows that such algorithms can be turned into black-box global optimization algorithms by using other types of targets distributions.

Algorithm \ref{alg:heavyCE} requires the computation of expectations with respect to $\pi_{\theta_k}^f$ at every iteration $k \in \mathbb{N}$. In practice, these expectations can be approximated with samples from the current proposal $p_{\theta_k}$ that are then weighted according to their rank, as we discussed in Section \ref{sec:33}. Such approximations are consistent when the number of samples goes to infinity, as discussed in Section \ref{sec:33} and \cite{ollivier2017}. In the context of computational statistics, algorithms similar to Algorithm \ref{alg:heavyCE} have been implemented for experiments in \cite{cappe2008}.

\section{Conclusion and perspectives}
\label{section:perspectives}

We have proposed in this work divergence-based conditions that imply the quantile improvement results achieved by the IGO framework, and can also be used to show improvements in terms of other expectation-based reformulations of the original problem. Therefore, our results can be seen as an alternative way to IGO, to prove that an algorithm achieves a quantile improvement result. The introduced divergence-based conditions are more general, in the sense that IGO algorithms satisfy  them, and our results further allow to predict the magnitude of the quantile improvement from the decrease in divergence. Our divergence-based conditions also allow to cover more general families of proposals than exponential families, including mixtures or heavy-tailed distributions.

In our proofs, we leveraged existing results from statistics and machine learning, related to divergence minimization in the context of variational inference. This connection between the two fields opens new perspectives for the design and study of black-box global optimization algorithms. Future works could exploit this connection to use more complex proposal distribution and adaptation strategies. In particular, existing schemes achieving a divergence-decrease can be turned into black-box global optimization algorithms by using an optimization-based target distribution, as we demonstrated for mixture and heavy-tailed proposals.

\bibliographystyle{abbrv}
\bibliography{biblio}

\end{document}